\documentclass{amsart}

\usepackage{amssymb}
\usepackage{amsthm}
\usepackage{wasysym}
\usepackage[numbers]{natbib}

\newcommand{\abs}[1]{\left|#1\right|}
\newcommand{\norm}[1]{\left\| #1 \right\|}

\newcommand{\A}{\mathcal{A}}
\newcommand{\xnice}{\mathbf{x}}

\newcommand{\Ho}{\mathcal{H}}

\newcommand{\CT}{\mathcal{C}}
\newcommand{\LT}{\mathcal{L}}
\newcommand{\fhat}{\hat{f}}

\newcommand{\Mtilde}{\tilde{M}}
\newcommand{\Fourier}{\mathcal{F}}
\newcommand{\R}{\mathbb{R}}
\newcommand{\C}{\mathbb{C}}
\newcommand{\N}{\mathbb{N}}

\DeclareMathOperator{\spt}{supp}

\newtheorem{theorem}{Theorem}[section]
\newtheorem{proposition}[theorem]{Proposition}
\newtheorem{corollary}[theorem]{Corollary}
\newtheorem{lemma}[theorem]{Lemma}

\theoremstyle{remark}
\newtheorem{remark}[theorem]{Remark}

\begin{document}

\title[A quantified Tauberian theorem and local decay of $C_0$-semigroups]{A quantified Tauberian theorem and local decay of $C_0$-semigroups}
\author[Reinhard Stahn]{Reinhard Stahn}

\begin{abstract} 
 We prove a quantified Tauberian theorem for functions under a new kind of Tauberian condition. In this condition we assume in particular that the Laplace transform of the considered function extends to a domain to the left of the imaginary axis, given in terms of an increasing function $M$ and is bounded at infinity within this domain in terms of a \emph{different} increasing function $K$. Our result generalizes \cite[Theorem 4.1]{BattyBorichevTomilov2016}. We also prove that the obtained decay rates are optimal for a very large class of functions $M$ and $K$. Finally we explain in detail how our main result improves known decay rates for the local energy of waves in odd-dimensional exterior domains.
\end{abstract}

\maketitle

{\let\thefootnote\relax}
{\let\thefootnote\relax}


\section{Introduction}\label{sec: Introduction}
 In the last decade there has been much activity in the field of \emph{quantified} Tauberian theorems for functions of a real variable \cite{LiuRao2005,BatkaiEngelPruessSchnaubelt2006,ChillTomilov2007,BattyDuyckaerts2008,BorichevTomilov2010,Martinez2011,BattyChillTomilov2016,ChillSeifert2016,BattyBorichevTomilov2016}. See also \cite{Seifert2015,Seifert2016} and references therein for quantified Tauberian theorems on sequences and \cite{Hartlapp2017} for Dirichlet series. We refer to \cite{Korevaar} and \cite[Chapter 4]{ArendtBattyHieberNeubrander} for a general overview on Tauberian theory.
 
 Let $X$ be a Banach space and $f:\R_+\rightarrow X$ be a locally integrable function. For some continuous and increasing function $M:\R_+\rightarrow [2,\infty)$ let us define
 \begin{equation}\nonumber
  \Omega_M = \left\{ z\in \C; 0 > \Re z > - \frac{1}{M(\abs{\Im z})} \right\} .
 \end{equation}
 The above mentioned articles impose essentially the Tauberian condition that the function $f$ has a bounded derivative (in the weak sense), the Laplace transform $\fhat$ extends across the imaginary axis to $\Omega_M$ and it satisfies a growth condition, \emph{also expressed in terms of $M$}  in $\Omega_M$ at infinity. The decay rate (the rate of convergence to zero) is then determined in terms of $M$. For example, a polynomially growing $M$ gives a polynomial decay rate and an exponentially growing $M$ gives a logarithmic decay rate. In general $\fhat$ could also have a finite number of singularities on the imaginary axis \cite{Martinez2011}, but we are not interested in this situation in the present article.
 
 The pioneering works \cite{LiuRao2005,BatkaiEngelPruessSchnaubelt2006} focus on polynomial decay for orbits of $C_0$-semigroups. A generalization for functions (as formulated above) and to arbitrary decay rates was given in \cite{BattyDuyckaerts2008} for the first time. There the authors also improved the decay rates from \cite{LiuRao2005,BatkaiEngelPruessSchnaubelt2006}. In \cite{BorichevTomilov2010} it was shown that the results of \cite{BattyDuyckaerts2008} are optimal in the case of polynomial decay. We want to emphasize at this point that the main result of \cite{BattyDuyckaerts2008} for the special case of a truncated orbit of a unitary group $U$ of operators (i.e. $f(t)=P_2 U(t) P_1$ for some bounded operators $P_1, P_2$) were already obtained in the earlier article \cite{PopovVodev1999} \emph{with the same rate of decay}. Actually the authors only formulated a theorem on polynomial decay but in the retrospective it is not difficult to generalize their proof to arbitrary decay rates.
 
 A major contribution to the field of Tauberian theorems is the recent article \cite{BattyBorichevTomilov2016}. The authors extended the known Tauberian theorems to $L^p$-rates of decay. On the basis of a technique already applied in \cite{BorichevTomilov2010} the authors showed the optimality of their results in the case of polynomial decay. Another important observation, made in \cite{BattyBorichevTomilov2016}, concerns the above mentioned growth condition. In \cite{BattyDuyckaerts2008} it was assumed that the norm of $\fhat(z)$ is bounded by $M(\abs{\Im z})$ in $\Omega_M$. This condition was weakened in \cite{BorichevTomilov2010} in case of polynomial decay, and later in \cite{BattyBorichevTomilov2016} assuming merely that $\fhat(z)$ can be bounded by a polynomial in $(1+\abs{\Im z})M(\abs{\Im z})$. 

 The aim of the present article is to further generalize the growth condition on $\fhat$ in $\Omega_M$. That is, we introduce a second continuous and increasing function $K:\R_+\rightarrow [2,\infty)$ and assume that the norm of $\fhat(z)$ is bounded by $K(\abs{\Im z})$ in $\Omega_M$. The decay rate is then given in terms of $M$ and $K$.
 
 Let $M^{-1}$ denote the right-continuous right-inverse of $M$ given by $M^{-1}(t)=\sup\{s\geq0; M(s)=t\}$ for all $t\geq0$. Let
\begin{equation}\nonumber
 w_M(t) = 
 \begin{cases}
  M^{-1}(t) & \text{ if } t\geq M(1) \\
  1 & \text{ else.}
 \end{cases}
\end{equation}
 We are now ready to state our first main result, a generalization of \cite[Theorem 4.1]{BattyBorichevTomilov2016}.

\begin{theorem}\label{thm: M_K theorem - Laplace variant}
 Let $(X,\norm{\cdot})$ be a Banach space, $m\in\N$, and $f:\R_+\rightarrow X$ be a locally integrable function such that its $m$-th weak derivative $f^{(m)}$ is in $L^p(\R_+; X)$ for some $1< p\leq \infty$. Assume that there exist continuous and increasing functions $M, K:\R_+\rightarrow [2,\infty)$ satisfying
 \begin{enumerate}
  \item[(i)] $\forall s>1: K(s) \geq \max\{s, M(s)\}$,
  \item[(ii)] $\exists\varepsilon\in(0,1): K(s) = O\left(e^{e^{(sM(s))^{1-\varepsilon}}}\right)$ as $s\rightarrow\infty$.
 \end{enumerate}
 such that the Laplace transform $\fhat$ of $f$ extends analytically to $\Omega_M\cup\C_+$ and
 \begin{equation}\label{eq: fhat condition}
  \norm{\fhat(z)} \leq K(\abs{\Im z}) \text{ for all } z\in\Omega_M .
 \end{equation}
 Then there exists a constant $c_1>0$ such that
 \begin{equation}\label{eq: rate}
  \left(t\mapsto \norm{w_{M_K}(c_1 t)^m f(t)}\right) \in L^p(\R_+),
 \end{equation}
 where $M_K(s):=M(s)\log(K(s))$.
\end{theorem}

 \begin{remark}\label{rem: fhat well-defined}
  Note that a function $f\in L^1_{loc}(\R_+; X)$ with $f^{(m)}\in L^p(\R_+;X)$ is polynomially bounded. In fact, $\norm{f(t)}\leq C (1+t)^{m-1/p}$ holds for all $t\geq0$. In particular the Laplace transform of $f$ is well-defined in the interior of $\C_+$ as an absolutely convergent integral.
 \end{remark}

 \begin{remark}
  One can drop condition (i) on $K$ but then one has to replace $M_K$ by the function given by $M(s)\log((2+s)M(s)K(s))$. 
 \end{remark}
 
 \begin{remark}
  We are not able to prove the theorem for $\varepsilon=0$ in condition (ii). In Section \ref{sec: remark on condition ii} the reader can find a short discussion on a slightly weaker constraint on $K$.
 \end{remark}

 If we replace $K(\abs{\Im z})$ in (\ref{eq: fhat condition}) by $((1+\abs{\Im z})M(\abs{\Im z}))^\alpha$ for some $\alpha>0$ and set $m=1$ we recover \cite[Theorem 4.1]{BattyBorichevTomilov2016}. Our theorem applies perfectly to local energy decay of waves in odd-dimensional exterior domains. Here $f$ is typically a spatially truncated orbit of a solution to the wave equation and one is often confronted with the situation that $M$ is constant and $K$ is asymptotically larger than any polynomial. In this situation no known Tauberian result applies directly. One might guess that one can apply the Phragm\'en-Lindel\"of principle to get a better estimate on $\fhat$ on a smaller domain to the left of the imaginary axis. Indeed this works, and as shown in \cite{BonyPetkov2006} one can apply known Tauberian theorems after this procedure. However in Section \ref{sec: Applications: decay of waves} we discuss the application to local decay of waves in exterior domains in detail and show that this procedure yields a weaker estimate than a direct application of Theorem \ref{thm: M_K theorem - Laplace variant}. 
 
We prove Theorem \ref{thm: M_K theorem - Laplace variant} as a corollary to the following variant which is a generalization of \cite[Theorem 2.1(b)]{ChillSeifert2016}:
\begin{theorem}\label{thm: M_K theorem - Fourier variant}
 Let $(X,\norm{\cdot})$ be a Banach space, $m\in\N$, and $f:\R_+\rightarrow X$ be a locally integrable function such that $f^{(m)}\in L^p(\R_+; X)$ for some $1< p\leq \infty$. Let $M$ and $K$ be as in Theorem \ref{thm: M_K theorem - Laplace variant}. Assume that the Fourier transform $F$ of $f$ is of class $C^{\infty}$ and its derivatives satisfy for all $j\in\N_0$
 \begin{equation}\label{eq: F condition}
  \norm{F^{(j)}(s)} \leq j!K(\abs{s}) M(\abs{s})^j \text{ for all } s\in\R .
 \end{equation}
 Then there exists a constant $c_1>0$ such that
 \begin{equation}\label{eq: rate - Fourier variant}
  \left(t\mapsto \norm{w_{M_K}(c_1 t)^m f(t)}\right) \in L^p(\R_+),
 \end{equation}
 where $M_K(s):=M(s)\log(K(s))$.
\end{theorem}

\begin{remark}
 Note that the Fourier transform of $f$ is well-defined in the sense of tempered distributions since $f$ is polynomially bounded (compare with Remark \ref{rem: fhat well-defined}).
\end{remark}

A theorem of this type (for $p=\infty$, $m=1$ and $K=M$) was formulated for the first time in \cite{ChillSeifert2016}. A main contribution of the authors was also to provide a new and easier to understand technique - on the basis of Ingham's original proof of the unquantified version \cite{Ingham1935} - for proving Tauberian theorems. For example in \cite{BattyDuyckaerts2008} and \cite{BattyBorichevTomilov2016} one main difficulty is to choose contours for integration in the complex plane in a clever way. In \cite{ChillSeifert2016} the authors avoid this technicality by considering the derivatives of the Fourier transform of $f$ instead of the Laplace transform.

To prove Theorem \ref{thm: M_K theorem - Fourier variant} we adapt the proof of \cite[Theorem 2.1(b)]{ChillSeifert2016}. That is - for $m=1$ - we decompose $f=[f-\phi_R*f]+\phi_R*f=J_1+J_2$ into two terms with the help of some suitably chosen and scaled convolution kernel $\phi_R(t)=R\phi(Rt)$ with $\int_{\R} \phi(t) dt=1$. Then we estimate the $X$-norm of $J_1(t,R)$ and $J_2(t,R)$ in terms of $R$ and $t$, solely assuming $f'\in L^p$ respectively the bounds on all derivatives $F^{(j)}$. Finally we optimize the sum of these two estimates by choosing $R=w_{M_K}(c_1t)$ for a sufficiently small $c_1$.

We improve the techniques of \cite{ChillSeifert2016} in the following way: We estimate $J_1(t,R)$ from above by a Poisson integral $R^{-1}P_{R^{-1}}*\norm{f'}(t)$ which makes it possible to apply a fundamental result on Carleson measures. We note that this technique was already applied in \cite{BattyBorichevTomilov2016}. 
Compared to the proof in \cite{ChillSeifert2016} we get a better estimate on $J_2(t,R)$ by choosing a better convolution kernel $\phi$. Also the Fourier transform $\psi$ of our convolution kernel is a $C_c^{\infty}$-function which simplifies the prove slightly. Our choice of $\psi$ is based on the Denjoy-Carleman theorem on quasi-analytic functions.

The paper is organized as follows. In Section \ref{sec: Proof of Theorem Fourier} and \ref{sec: proof of Theorem Laplace} we prove Theorem \ref{thm: M_K theorem - Fourier variant} and \ref{thm: M_K theorem - Laplace variant}, respectively. In Section \ref{sec: Optimality} we prove the optimality of Theorem \ref{thm: M_K theorem - Laplace variant} for a very large class of possible choices of $M$ and $K$. This is even new in the case where $K=M$. To prove the optimality we make a similar construction as in \cite{BorichevTomilov2010}. As a side product this construction also shows that there actually exist functions $f$ satisfying the hypotheses of Theorem \ref{thm: M_K theorem - Laplace variant} for $K$ increasing faster than any polynomial in $sM(s)$, but do not satisfy (\ref{eq: fhat condition}) if one replaces $K$ by a polynomial in $sM(s)$ (see Remark \ref{rem: YY}). This proves that Theorem \ref{thm: M_K theorem - Laplace variant} is a \emph{proper} generalization of \cite[Theorem 4.1]{BattyBorichevTomilov2016}. A short discussion on the optimal choice of $c_1$ in (\ref{eq: rate}) is included in Subsection \ref{sec: Optimality of c1}. In Subsection \ref{sec: Decay of C0-semigroups} we explain how to get local decay rates for $C_0$-semigroups from our results. Finally in Subsection \ref{sec: Local energy decay} we apply this to local energy decay of waves in odd-dimensional exterior domains.

\subsection{Notation}
We denote $\R_+=[0,\infty)$ and $\C_+=\{z\in\C; \Re z \geq 0\}$. By $\N_0$ we denote the natural numbers including $0$. For $m\in\N_0$ we define $\N_m$ to be the natural numbers greater or equal to $m$. By $C$ we denote a strictly positive constant which may change implicitly their value from line to line. Every statement in this article which includes $C$ remains true if one replaces $C$ by a larger constant. Other strictly positive constants, having the names $C_1, C_2,\ldots$ are not allowed to change their values - except it is explicitly stated. Analogously $c,c_1,c_2,\ldots$ are strictly positive constants which might be replaced by smaller constants without invalidating any statement in our article. We say that a function $\phi:\R\rightarrow\R$ decays \emph{rapidly} if for any $n\in\N_0$ there exists a constant $C$ such that $\abs{\phi(t)}\leq C(1+t)^{-n}$.


\section{Proof of Theorem \ref{thm: M_K theorem - Fourier variant}}\label{sec: Proof of Theorem Fourier}
Without loss of generality we may assume that $f(0)=f'(0)=\ldots=f^{(m-1)}(0)=0$. If this was not satisfied we could replace $f$ by $f-g$ for some function $g\in C_c^m([0,t_1); X)$ with $g(0)=f(0), \ldots, g^{(m-1)}(0)=f^{(m-1)}(0)$ and $t_1>0$ arbitrary. This neither changes the asymptotics of $f$ at infinity nor does it change the growth of $F$ and its derivatives at infinity considerably. To see this note that the Fourier transform $G$ of $g$ satisfies 
\begin{equation}\nonumber
 \norm{G^{(j)}(s)}\leq t_1^{j+1}\norm{g}_{\infty} \text{ for } j\in\N_0 \text{ and } s \in \R .
\end{equation}
Now let us extend $f$ by zero on the negative numbers. By our additional assumptions we see that the extended function is $(m-1)$-times continuously differentiable on the whole real line and $f^{(m)}\in L^p(\R;X)$.

Let $\psi \in C_c^{\infty}(\R)$ with $\spt\psi\subseteq[-1,1]$ and $\psi(0)=1$ be a function to be fixed later in the proof. Let
\begin{equation}\nonumber
 \phi(t) = \Fourier^{-1}\psi (t) = \frac{1}{2\pi} \int_{-\infty}^{\infty} e^{ist} \psi(s) ds 
\end{equation}
be its inverse Fourier transform. Note that $\phi$ is a Schwartz function with $\int_{-\infty}^{\infty}\phi dt=\psi(0)=1$. For $R>0$ let $\phi_R(t)=R\phi(Rt)$ and $\psi_R(s)=\psi(s/R)$. Let us decompose
\begin{align*}
 f(t) &= (\delta-\phi_R)^{*m}*f(t) - [(\delta-\phi_R)^{*m} - \delta]*f(t) \\
 &= \left[ \sum_{j=0}^m \binom{m}{j}(-1)^j \phi^{*j}_R*f \right](t) - \left[ \sum_{j=1}^m \binom{m}{j}(-1)^j \phi^{*j}_R*f \right](t) \\
 &=: J_1(t,R) + J_2(t,R).
\end{align*}
Here by $\phi^{*j}$ we denote the $j$-times convolution of $\phi$ with itself. We also define $\phi^{*0}=\delta$ (delta-function). Note that $(\phi_R)^{*j}=(\phi^{*j})_R$.

\subsection{Estimation of $J_1$}
Let us define the Poisson kernel by
\begin{equation}\nonumber
 P_y(t) = \frac{1}{\pi}\cdot \frac{y}{t^2+y^2}.
\end{equation}
 Recall that by Young's inequality the Poisson kernel acts as a continuous operator on $L^p(\R)$ via convolution.
\begin{lemma}\label{lem: J1 estimate}
 Let $1\leq p\leq\infty$ and $m\in\N_1$. Let $f:\R\rightarrow X$ be a locally integrable function such that $f^{(m)}\in L^p(\R; X)$. Let $\phi$ be as above. Then there exists a constant $C>0$ (only depending on $\phi$ and $m$) such that
 \begin{equation}
  \norm{(\delta-\phi_R)^{*m}*f(t)} \leq \frac{C}{R^m} P_{\frac{1}{R}}*\norm{f^{(m)}}(t)
 \end{equation}
 holds for all $t\geq0$ and $R>0$.
\end{lemma}

\begin{remark}
 It is clear from the proof that in the statement of the lemma one can replace $P=P_1$ by any positive and integrable kernel bounded from below by $c(1+t)^{-\alpha}$ for some $\alpha>1$. We then define $P_y(t)=y^{-1}P(y^{-1}t)$. Unfortunately this is not consistent with the definition of $\phi_R$, but for the Carleson measure argument below it is more convenient to define $P_y$ as above.
\end{remark}

\begin{proof}
 Let us define two antiderivatives of $\phi$
 \begin{equation}\nonumber
  \Phi_-(t) = \int_{-\infty}^{t}\phi(\tau)d\tau, \quad \Phi_+(t) = -\int_{t}^{\infty}\phi(\tau)d\tau .
 \end{equation}
 Furthermore we define the following auxiliary function
 \begin{equation}\label{eq: Phi}
  \Phi(t) = 
  \begin{cases}
   \Phi_-(t) & \text{ if } t<0 \\
   \Phi_+(t) & \text{ if } t\geq 0
  \end{cases}
   .
 \end{equation}
 We observe that the derivative of $\Phi$ is $\phi$ plus a factor times the delta function at zero. This observation is the reason why we split the integral from the following calculation at $0$. 
 
 First we consider the case $m=1$.
 \begin{align}\nonumber
  [f-\phi_R*f](t) &= \int_{-\infty}^{\infty} (f(t)-f(t-\tau))\phi_R(\tau) d\tau \\ \nonumber
  &= \left[\int_{-\infty}^{0} + \int_{0}^{\infty}\right] (f(t)-f(t-\frac{\tau}{R}))\phi(\tau) d\tau \\ \nonumber
  &= -\frac{1}{R} \int_{-\infty}^{\infty} f'(t-\frac{\tau}{R})\Phi(\tau) d\tau \\ \label{eq: induction beginning}
  &= -\frac{1}{R} \Phi_R*f'(t) .
 \end{align}
 We need to explain why the partial integration executed from line two to three produces no boundary terms at $-\infty,0$ and $\infty$. At zero there are no boundary terms since $(f(t)-f(t-\frac{\tau}{R}))$ vanishes at $\tau=0$ and the two limits $\lim_{t\rightarrow0\pm}\Phi(t)$ exist. Recall that $f$ is polynomially bounded. Moreover the function $\Phi$ decays rapidly at infinity. Thus there are no boundary terms at plus or minus infinity.
 Finally the last equality together with the fact that $\Phi$ decays rapidly implies
 \begin{align*}
  \norm{[f-\phi_R*f](t)} &\leq \frac{C}{R} \int_{-\infty}^{\infty} \norm{f'(t-\frac{\tau}{R})}\frac{1}{\tau^2+1} d\tau \\
  &\leq \frac{C}{\pi R} \int_{-\infty}^{\infty} \norm{f'(t-\tau)}\frac{R^{-1}}{\tau^2+R^{-2}} d\tau \\
  &= \frac{C}{R} P_{\frac{1}{R}}*\norm{f'}(t) .
 \end{align*}
 
 Now we consider the case $m\in\N_2$. Let us define recursively $f_{j+1}=f_j-\phi_R*f_j, f_0=f$ for $j\in\{0,1,\ldots,m-1\}$. Clearly $f_m=(\delta-\phi_R)^{*m}*f$. We prove now $f_j=(-1/R)^j\Phi_R^{*j}*f^{(j)}$ via induction on $j$. Observe that for any $j\in\N_1$ the function $\Phi^{*j}$ decays rapidly. For $j=1$ the inductive hypothesis is precisely (\ref{eq: induction beginning}). Assume that the hypothesis is valid for some $j<m$. Then by (\ref{eq: induction beginning}) for $f$ replaced by $f_j$
 \begin{align}\nonumber
  f_{j+1} = f_j - \phi_R*f_j = -\frac{1}{R}\Phi_R*f_j' = \left( -\frac{1}{R} \right)^{j+1}\Phi_R^{*(j+1)}*f^{(j+1)} .
 \end{align}
 From here we can finish the proof as in the case $m=1$.
\end{proof}

Since the $L^1$-norm of the Poisson kernel is $1$ (for any $y>0$) we see from Young's inequality that for any $g\in L^p(\R)$ and $y>0$ it holds that $\norm{P_y*g}_{L^p}\leq \norm{g}_{L^p}$. If $p=\infty$ and if we set $R=R(t)=w_{M_K}(c_1 t)$ we deduce from Lemma \ref{lem: J1 estimate} that 
\begin{equation}\label{eq: Carleson estimate for f' p=infty}
 R(t)^m \norm{(\delta-\phi_{R(t)})^{*m}*f(t)} \leq C_{c_1} < \infty 
\end{equation}
holds for all $t\geq 0$. If we compare this with (\ref{eq: rate - Fourier variant}) we see that this already yields the desired estimate on $J_1$ in the case $p=\infty$. If $p<\infty$ we need a slightly more involved argument based on a property of Carleson-measures.
 
Therefore let $P*g(t,y):=P_y*g(t)$ and let $\mu$ be a Borel measure on the upper half-plane $H=\{(t,y)\in\R^2; y>0\}$. Now we ask for which measures $\mu$ an inequality 
\begin{equation}\label{eq: Carleson estimate}
 \norm{P*g}_{L^p(H,d\mu)} \leq C_p \norm{g}_{L^p(\R)}
\end{equation}
holds for all $g\in L^p(\R)$ with a constant $C_p$ not depending on $g$? Note that the inequality $\norm{P_y*g}_{L^p}\leq \norm{g}_{L^p}$ is a special case of (\ref{eq: Carleson estimate}) for $C_p=1$ with $\mu$ being the one-dimensional Hausdorff measure of the line $\{(t,y)\in H; t\in\R\}\subset H$. Actually for $1<p<\infty$ one can characterize the class of all measures $\mu$ for which (\ref{eq: Carleson estimate}) holds for all $g$. These measures are called \emph{Carleson measures} (see \cite[Theorem I.5.6.]{Garnett1981}). Let $\gamma:\R\rightarrow (0,\infty)$ be a bounded continuous function with bounded variation.  Then the one-dimensional Hausdorff measure of
\begin{equation}\nonumber
 \Gamma = \{t+\gamma(t); t\in\R\} \subset H
\end{equation}
is a Carleson measure. Now let $\gamma(t)=1/R(t)=1/w_{M_K}(c_1 t)$ for $t>0$ and $\gamma(t)=w_{M_K}(0)=1$ for $t<0$. If we set $\mu_{M_K}$ to be the Carleson measure corresponding to this particular choice of $\gamma$ then we deduce that for $1<p<\infty$
\begin{equation}\label{eq: Carleson estimate for f'}
 \norm{P*\norm{f^{(m)}}}_{L^p(H,d\mu_{M_K})} \leq C_p \norm{f^{(m)}}_{L^p(\R_+; X)} < \infty.
\end{equation}
From this together with Lemma \ref{lem: J1 estimate} we deduce
\begin{lemma}\label{lem: final J1 estimate} Let $c_1>$ and define $R(t)=w_{M_K}(c_1t)$.
 (i) Then for $p=\infty$ we have
 \begin{equation}\nonumber
  \sup_{0<t<\infty} R(t)^m \norm{ (\delta-\phi_{R(t)})^{*m}*f(t) } \leq C \norm{f^{(m)}}_{L^{\infty}(\R_+;X)} .
 \end{equation}
 (ii) For $1<p<\infty$ we have
 \begin{equation}\nonumber
  \int_0^{\infty} \norm{ R(t)^m (\delta-\phi_{R(t)})^{*m}*f(t) }^p dt \leq C \norm{f^{(m)}}_{L^p(\R_+;X)}^p .
 \end{equation}
 In both cases $C$ does not depend on $f$.
\end{lemma}

\subsection{Estimation of $J_2$}
The following Lemma is only necessary if $p\neq\infty$.
\begin{lemma}\label{lem: K auxiliary estimate}
 There exists a $\delta>0$ such that $K(w_{M_K}(t))\geq t^{\delta}$ for all $t\geq M_K(1)$.
\end{lemma}
\begin{proof}
 Let $R=w_{M_K}(t)$. Since $w_{M_K}$ is essentially the right-inverse of $M_K$ we have
 \begin{align*}
  t = M_K(R) = M(R)\log(K(R))\geq M(R)\log(M(R)).
 \end{align*}
 The inverse of the function $x\mapsto x\log(x)$ is asymptotically equal to $y\mapsto y/\log(y)$ for large $y>0$. Hence there exists a $\delta>0$ such that $M(R) \leq \delta^{-1} t/\log(t)$. Thus
 \begin{align*}
  K(R) = \exp(\log(K(R))) = \exp\left(\frac{t}{M(R)}\right) \geq \exp(\delta\log(t)) = t^{\delta}.
 \end{align*}
\end{proof}
At this point in the proof we fix a $\psi$ having one additional property. We assume that the derivatives of $\psi$ satisfy for some $C_1>0$
\begin{equation}\label{eq: psi Denjoy-Carleman condition}
 \forall j\in\N_0: \sup_{s\in[-1,1]}\abs{\psi^{(j)}(s)} \leq C_1^{j+1}A_j \text{ with } A_j = (j\log(2+j)^{1+\varepsilon})^j .
\end{equation}
Note that (\ref{eq: psi Denjoy-Carleman condition}) can not be satisfied by any $\psi$ if we would replace $A_j$ by $j!$ since then $\psi$ would be analytic and hence can not have compact support and $\psi(0)=1$ at the same time. The Denjoy-Carleman\footnote{A special version of the Denjoy-Carleman theorem (sufficient for our considerations) reads as follows: Let $S$ be the set of $C^{\infty}$-functions on $\R$ supported on $[-1,1]$ such that (\ref{eq: psi Denjoy-Carleman condition}) holds for a sequence $(A_j)$ such that $(\sqrt[j]{A_j})$ is increasing. Then $S$ contains a non-zero function if and only if $\sum_{j} 1/\sqrt[j]{A_j}<\infty$.} theorem (see e.g. \cite[Theorem 1.3.8]{HoermanderI} or \cite{Cohen68}) gives a description of those sequences $(A_j)$ which allow for compactly supported non-zero functions $\psi$ satisfying the inequality in (\ref{eq: psi Denjoy-Carleman condition}). In particular, the Denjoy-Carleman theorem implies that our choice of $A_j$ is admissible for the existence of such a $\psi$. Conversely it implies that there is no $\psi\in C_c^{\infty}(\R)\backslash\{0\}$ which satisfies (\ref{eq: psi Denjoy-Carleman condition}) with $\varepsilon=0$.

Now we proceed with the estimation of $J_2(t,R)$. Therefore we have to estimate $J_{2,j}(t,R)=\phi_R^{*j}*f(t)$ for $j\in\{1,\ldots,m\}$. First let us consider $J_{2,1}$. Let $N\in\N_0$. Integration by parts $N$-times yields
\begin{align}\label{eq: J_2,1 on the Fourier side}
 J_{2,1}(t,R) &= \frac{1}{2\pi} \int_{-\infty}^{\infty} e^{ist}F(s)\psi_R(s) ds \\ \nonumber
 &= \frac{1}{2\pi} \left(\frac{i}{t}\right)^N \int_{-R}^{R} e^{ist}\left[\sum_{j=0}^N \binom{N}{j} F^{(N-j)} R^{-j}(\psi^{(j)})_R\right](s) ds .
\end{align}
To verify the following calculations recall (\ref{eq: F condition}) and (\ref{eq: psi Denjoy-Carleman condition}). We estimate the integral very roughly from above: length of interval of integration times supremum of the integrand within this interval. We also use Stirling's formula implying for example that $(cj)^j\leq j!\leq (Cj)^j$ for appropriate constants $c,C>0$.
\begin{align}\nonumber
 \norm{R^mJ_{2,1}(t,R)} &\leq Ct^{-N} R^{m+1} \sum_{j=0}^N \binom{N}{j} (N-j)! K(R) M(R)^{N-j} \left(\frac{\norm{\psi^{(j)}}_{\infty}^{\frac{1}{j}}}{R}\right)^j \\ \label{eq: no observation}
 &\leq C \cdot R^{m+1} K(R)\left(\frac{C_2M(R)N}{e t}\right)^N \cdot \sum_{j=0}^N\left(\frac{C_3\log(2+N)^{1+\varepsilon}}{RM(R)}\right)^j \\ \nonumber
 &=: C \cdot A \cdot B.
\end{align}
The second inequality is valid for sufficiently large $C_2, C_3>0$. Now let us set $N=\lfloor t/(C_2M(R)) \rfloor$ and $R=w_{M_K}(c_1 t)$. The constant $c_1>0$ will be chosen later. Then the condition (ii) on $K$ implies
\begin{equation}\nonumber
 B \leq \sum_{j=0}^N\left(\frac{C_3\log((c_1C_2)^{-1}\log(K(R)))^{1+\varepsilon}}{RM(R)}\right)^j 
 \leq \sum_{j=0}^N\left(\frac{C_4(RM(R))^{1-\varepsilon^2}}{RM(R)}\right)^j
 \leq C .
\end{equation}
The constant in the last inequality does not depend on $t$. Moreover
\begin{align*}
 A \leq C R^{m+1} K(R) e^{-N} \leq C R^{m+1} K(R) e^{-\frac{\log(K(R))}{c_1C_2}} = C R^{m+1} K(R)^{1-\frac{1}{c_1C_2}} .
\end{align*}
If we choose $c_1$ sufficiently small Lemma \ref{lem: K auxiliary estimate} implies that
\begin{align}\label{eq: J2 final estimate}
 \norm{w_{M_K}(c_1 t)^mJ_{2,1}(t,w_{M_K}(c_1 t))} \leq
 \begin{cases}
  C & \text{ if } p=\infty, \\
  \frac{C}{(1+t)^{2/p}} & \text{ if } 1\leq p < \infty .
 \end{cases}
\end{align}
Clearly (\ref{eq: J_2,1 on the Fourier side}) remains valid if one replaces $J_{2,1}$ by $J_{2,k}$ and $\psi$ by its $k$-th power $\psi^k$. It is not difficult to check that $\psi^k$ also satisfies (\ref{eq: psi Denjoy-Carleman condition}) if one replaces $C_1^{j+1}$ by $C_1^k(kC_1)^{j}$. Therefore (\ref{eq: J2 final estimate}) remains true after replacing $J_{2,1}$ by $J_2$. This together with Lemma \ref{lem: final J1 estimate} proves Theorem \ref{thm: M_K theorem - Fourier variant}.

\subsection{A remark on condition (ii) for $K$}\label{sec: remark on condition ii}
Our proof breaks down if we allow $\varepsilon$ to be zero in condition (ii) in Theorem \ref{thm: M_K theorem - Laplace variant} (and \ref{thm: M_K theorem - Fourier variant}). This is essentially due to the fact that by the Denjoy-Carleman theorem a function $\psi$ satisfying (\ref{eq: psi Denjoy-Carleman condition}) for $\varepsilon=0$ is necessarily \emph{quasi-analytic}. This means that $\psi^{(j)}(s_0)=0$ for a single $s_0\in\R$ but all $j\in\N$ automatically implies $\psi=0$. However, one can
weaken (ii) slightly by choosing for some given $\varepsilon\in(0,1)$ and $n\in\N_1$
\begin{align*}
 A_j &= j\cdot L_1(j)\cdot L_2(j)\cdot \ldots \cdot L_n(j)\cdot L_{n+1}(j)^{1+\varepsilon} \text{ with } \\
 L_k(j) &= \underbrace{[\log \circ \ldots \circ \log]}_{k \text{ times}} (1+k+j) .
\end{align*}
This allows to replace (ii) by the condition
\begin{equation}\nonumber
 K(s) = O\left(\exp\left(\exp\left( \frac{sM(s)}{L_1(sM(s))\cdot\ldots\cdot L_{n-1}(sM(s))\cdot L_n(sM(s))^{1+\varepsilon}} \right)\right)\right) .
\end{equation}
Again choosing $\varepsilon=0$ is forbidden for any $n$.

\section{Proof of Theorem \ref{thm: M_K theorem - Laplace variant}}\label{sec: proof of Theorem Laplace}
Lemma \ref{lem: Fourier vs Laplace} below implies that Theorem \ref{thm: M_K theorem - Laplace variant} and Theorem \ref{thm: M_K theorem - Fourier variant} are equivalent. To prepare the formulation of this lemma we introduce some notation. Let $M_1$, $M_2$, $K_1$, $K_2:\R_+\rightarrow [2,\infty)$ be continuous and increasing functions. For $f:\R_+\rightarrow X$ measurable and polynomially bounded and extended by zero on the negative real numbers we consider two distinct conditions. The first one is
\begin{align}\label{eq: Laplace condition}
 \forall z\in\Omega_{M_1}: \norm{\fhat(z)} \leq K_1(\abs{\Im z}).
\end{align}
This condition implicitly states that the Laplace transform of $f$ can be extended to $\Omega_{M_1}$. Let $F$ be the Fourier-transform of $f$. The second condition is
\begin{align}\label{eq: Fourier condition}
 \forall j\in\N_0, s\in\R: \norm{F^{(j)}(s)} \leq j! K_2(\abs{s})M_2(\abs{s})^j.
\end{align}
This condition implicitly states that the Fourier transform is a $C^{\infty}$-function. 

The following lemma relates these conditions to each other under a mild condition on $f$.

\begin{lemma}\label{lem: Fourier vs Laplace}
 Let $f:\R_+\rightarrow X$ be a measurable and polynomially bounded function with $f^{(m)}\in L^p(\R_+;X)$ for some $1\leq p\leq\infty$ and $m\in\N_1$. We extend $f$ by zero on the negative real numbers and denote by $F$ its Fourier transform. (a) If $F$ satisfies (\ref{eq: Fourier condition}) then $f$ satisfies (\ref{eq: Laplace condition}) with
 \begin{align*}
  M_1(s) = (1-\varepsilon)^{-1}M_2(s) \text{ and }
  K_1(s) = \varepsilon^{-1} K_2(s)
 \end{align*}
 for any $\varepsilon\in(0,1)$.
 (b) If $f$ satisfies (\ref{eq: Laplace condition}) then $F$ satisfies (\ref{eq: Fourier condition}) with
 \begin{align*}
  M_2(s) &= M_1\left(s+\frac{1}{M_1(s)}\right) \text{ and }  \\
  K_2(s) &= K_1\left(s+\frac{1}{M_1(s)}\right) + C_f\frac{M_1\left(s+\frac{1}{M_1(s)}\right)^{2-\frac{1}{p}}}{(1+s)^m} + C'_f .
 \end{align*}
 The constant $C_f$ depends only on $\norm{f^{(m)}}_{L^p}$, the constant $C'_f$ depends only on $\norm{f(0)}, \ldots,\norm{f^{(m-1)}(0)}$.
\end{lemma}
Before proving this lemma we finish the proof of Theorem \ref{thm: M_K theorem - Laplace variant}. Since $f$ satisfies (\ref{eq: Laplace condition}) for $M_1=M$ and $K_1=K$, Lemma \ref{lem: Fourier vs Laplace} implies that (\ref{eq: Fourier condition}) is true for $M_2$ and $K_2$ given as in part (b) of the lemma. In the following we assume $s>0$ large enough to satisfy $1/M_1(s)\leq s$. Note that condition (i) in Theorem \ref{thm: M_K theorem - Laplace variant} implies the existence of a (small) constant $c>0$ such that (for large $s$)
 \begin{equation}\nonumber
  c M_2(s)\log(K_2(s)) \leq M(2s)\log(K(2s)) .
 \end{equation}
This immediately yields for large $t$
 \begin{equation}\nonumber
   w_{M_K}(ct) \leq 2 w_{(M_2)_{K_2}}(t).
 \end{equation}
Therefore $w_{(M_2)_{K_2}}(c_1\cdot)^m f \in L^p$ for some $c_1>0$ implies that $ w_{M_K}(cc_1 \cdot)^mf \in L^p$. The proof of Theorem \ref{thm: M_K theorem - Laplace variant} is complete.


\begin{proof}[Proof of Lemma \ref{lem: Fourier vs Laplace}]
 Let us begin with the easier part (a). Hadamard's formula shows that (\ref{eq: Fourier condition}) implies that $\fhat$ is analytic in $\Omega_{M_2}\supset \Omega_{M_1}$. Let $z\in\Omega_{M_1}$ and let $s=\Im z$. Then 
 \begin{align*}
  \norm{\fhat(z)} = \norm{\sum_{j=0}^{\infty} \frac{1}{j!}\fhat^{(j)}(is)(z-is)^j} 
  \leq \sum_{j=0}^{\infty} K_2(s) M_2(s)^j\left(\frac{1-\varepsilon}{M_2(s)}\right)^j 
  = \varepsilon^{-1} K_2(s) .
 \end{align*}

 Let us now prove part (b). Let us fix $s\in\R$, let $r=1/M_1(\abs{s}+1/M(\abs{s}))$ and let $\gamma$ be the positively oriented circle of radius $r$ around $is$ in the complex plane. Note that $\gamma$ is indeed included in the closure of the union of $\Omega_{M_1}$ and $\C_+$. Let $\gamma_+$ and $\gamma_-$ be the intersection of $\gamma$ with $\C_+$ and $\C_-$, respectively. By Cauchy's formula we have
 \begin{align*}
  \fhat^{(j)}(is) &= \frac{j!}{2\pi i} \left[\int_{\gamma_-} + \int_{\gamma_+}\right] \frac{\fhat(z)}{(z-is)^{k+1}}\left(1+\frac{(z-is)^2}{r^2}\right) dz \\
  &=: j!\left[I_- + I_+\right] .
 \end{align*}
 Let us first estimate $I_-$:
 \begin{align}\nonumber
  \norm{I_-} &\leq \frac{1}{2\pi} \cdot r^{-j-1} \sup_{z\in\gamma_-} \norm{\fhat(z)} \cdot \pi r\cdot 2 \\ \label{eq: F vs L minus}
  &\leq K_1\left(\abs{s}+\frac{1}{M_1(\abs{s})}\right) M_1\left(\abs{s}+\frac{1}{M_1(\abs{s})}\right)^j .
 \end{align}
 Let us now estimate $I_+$:
 \begin{align*}
  I_+ &= \frac{1}{2\pi i}\int_{\gamma_+} \frac{\left(1+\frac{(z-is)^2}{r^2}\right)}{(z-is)^{j+1}} 
  \left( \sum_{k=0}^{m-1}z^{-j-1}f^{(k)}(0) + z^{-m} \int_0^{\infty} e^{-zt} f^{(m)}(t) dt \right) dz \\
  &=: \sum_{k=0}^{m-1} I_{+,k} + I_{+,m}
 \end{align*}
 It is an easy exercise to show that the integral of $e^{-rt\cos(\theta)}\cos(\theta)$ over $\theta\in(\pi/2,\pi/2)$ can be estimated from above by a constant times $((rt)^2+1)^{-1}$. Therefore by H\"older's inequality we get for large $\abs{s}$
 \begin{align*}
  \norm{I_{+,m}} &\leq \frac{C}{\abs{s}^m r^{j+1}} \int_0^{\infty} \int_{-\frac{\pi}{2}}^{\frac{\pi}{2}} e^{-rt\cos(\theta)}\cos(\theta) d\theta \norm{f^{(m)}(t)} dt \\
  &\leq \frac{C}{\abs{s}^m r^{j+2-1/p}} \norm{f^{(m)}}_{L^p} \\
  &\leq \frac{C\norm{f^{(m)}}_{L^p}}{\abs{s}^m} M_1(\abs{s}+1/M_1(\abs{s}))^{j+2-1/p} .
 \end{align*}
 A similar (and easier) estimate is true for the other summands $I_{+,k}$. This together with (\ref{eq: F vs L minus}) yields the claim.
\end{proof}

 \section{Optimality of Theorem \ref{thm: M_K theorem - Laplace variant}}\label{sec: Optimality}
 In this section we show that under the assumptions of Theorem \ref{thm: M_K theorem - Laplace variant} and for $p=\infty, m=1$ one can - up to improvement of the constant $c_1$ - not get a faster decay rate than the one already given by the theorem. To show this we use almost the same method as in \cite{BorichevTomilov2010}. There the authors showed the optimality in the very particular case that $M(s)=C(1+s^{\alpha})$ and $K(s)=C(1+s^\beta)$ for $\beta>\alpha/2>0$.
 
 \begin{theorem}\label{thm: Optimality}
  Let $c_1>0$ and let $M,K:\R_+\rightarrow[2,\infty)$ be continuous and increasing functions satisfying for some increasing function $N:\R_+\rightarrow[1,\infty)$
  \begin{enumerate}
   \item[(i)] $\lim_{s\rightarrow\infty} \frac{M_K(s)}{\log(2+s)} = \infty $ and $\exists\varepsilon>0, s_0>0 \forall s\geq s_0: K(s)\geq s^{\varepsilon}$,
   \item[(ii)] $\exists s_0>0\forall s\geq s_0,s'\geq0: M(s+s')\leq N(s') M(s)$.
  \end{enumerate}
  Then there exists a real number $\gamma\geq0$, not depending on $c_1$ and a locally integrable function $f:\R_+\rightarrow\C$ with $f'\in L^{\infty}(\R_+)$ such that
  \begin{equation}\label{eq: optimality fhat estimate}
   \abs{\fhat(z)} \leq \frac{C}{R} M(\abs{\Im z})^{\frac{1}{2}} K(\abs{\Im z})^{\frac{\gamma}{c_1}} \text{ for all } z\in\Omega_M
  \end{equation}
  and 
  \begin{equation}\label{eq: optimality decay rate}
   \limsup_{t\rightarrow\infty} M_K^{-1}(c_1t)\abs{f(t)} \geq c > 0.
  \end{equation}
  If instead of (ii) we have the stronger assumption that there exists a $\gamma_0\geq1$ such that
    \begin{enumerate}
   \item[(ii')] $\forall s_1>0 \exists s_0>0 \forall s\geq s_0,s'\leq s_1: M(s+s')\leq \gamma_0 M(s)$
  \end{enumerate}
  and if $\gamma>\gamma_0$ then it is possible to choose $f$ in such a way that  (\ref{eq: optimality fhat estimate}) holds for this choice of $\gamma$. If in addition $M$ is unbounded then it is possible to choose $f$ in such a way that (\ref{eq: optimality fhat estimate}) holds for all $\gamma>\gamma_0$.
 \end{theorem}
 
 \begin{remark}
  Note that condition (i) is only a very mild restriction. In fact, a typical situation where (i) is violated is that $M$ is a constant and $K$ grows at most polynomially. But then Theorem \ref{thm: M_K theorem - Laplace variant} implies exponential decay for $f$. This in turn implies, that the integral which defines $\fhat$ is absolutely convergent in a small strip to the left of the imaginary axis. In particular $\fhat$ extends analytically to this strip and is bounded there. So our results are trivially optimal in that case.
 \end{remark}

 Before we prove the Theorem we need a similar lemma as in \cite{BorichevTomilov2010}. Given a compactly supported measure $\mu$ on $\C\backslash\overline{\Omega_M\cup\C_+}$ we use the following notation for $z\in \Omega_M\cup\C_+$ and $t\geq0$
 \begin{align*}
  \CT\mu(z) = \int \frac{1}{z-\zeta} d\mu(\zeta),\, 
  \LT\mu(t) = \int e^{t\zeta} d\mu(\zeta),\,
  \LT'\mu(t) = \int \zeta e^{t\zeta} d\mu(\zeta).
 \end{align*}
 To simplify the notation we extend $M$ and $K$ symmetrically to the negative real axis.
 
 \begin{lemma}\label{lem: Optimality}
  Let $c_1,M$ and $K$ be as in Theorem \ref{thm: Optimality}. There exists a $\delta>0$ and $\gamma>0$, only depending on $M$ and $\delta$, such that for all $\varepsilon>0$ and $k_0\in\N_0$ there exists $k\in\N_{k_0}$ and a compactly supported Borel measure $\mu$ on $\C\backslash\overline{\Omega_M\cup\C_+}$ such that
  \begin{align}\label{eq: optimality 1}
   \abs{\CT\mu(z)}  \leq \frac{C}{R} M^{\frac{1}{2}}K^{\gamma} 1_{[R-2\delta,R+2\delta]}(\Im z) + \varepsilon, \\ \label{eq: optimality 2}
   \abs{\LT'\mu(t)} \leq C 1_{[\frac{k}{2\delta},\frac{2k}{\delta}]}(t) + \varepsilon, \\ \label{eq: optimality 3}
   \abs{\LT\mu(t)}  \leq \frac{C}{R} 1_{[\frac{k}{2\delta},\frac{2k}{\delta}]}(t) + \frac{\varepsilon}{\max\{R,M_K^{-1}(c_1 t)\}}, \\ \label{eq: optimality 4}
   \abs{\LT\mu(\frac{k}{\delta})} \geq \frac{c}{R}
  \end{align}
  holds for all $z\in\Omega_M$ and $t\geq0$. Here $R$ is the largest real number such that $c_1k=\delta M_K(R)$.   If instead of (ii) we have the stronger assumption that there exists a $\gamma_0\geq1$ such that
  \begin{enumerate}
   \item[(ii')] $\forall s_1>0 \exists s_0>0 \forall s\geq s_0,s'\leq s_1: M(s+s')\leq \gamma_0 M(s)$
  \end{enumerate}
  and if $\gamma>\gamma_0$ then it is possible to choose $f$ in such a way that (\ref{eq: optimality 1}) holds for this choice of $\gamma$. If in addition $M$ is unbounded then it is possible to choose $f$ in such a way that (\ref{eq: optimality 1}) holds for all $\gamma>\gamma_0$.
 \end{lemma}
 
 \begin{remark}\label{rem: Y}
  For $\Im z = R$ the inequality (\ref{eq: optimality 1}) holds also in the reverse direction (for a different value of $C$). This will be indicated in the proof.
 \end{remark}

 \begin{proof}
  Let $\delta>1/M(0)$ be a real number to be fixed later. Let $k\in\N_{k_0}$ to be fixed later. Let us define
  \begin{align*}
   w = iR - \delta,\, q = e^{2\pi i/(k+1)},\, \delta A = kl(k)
  \end{align*}
  where $l:\R_+\rightarrow(0,\infty)$ is a strictly increasing function such that $l(t)\geq\beta\log(e+t)$ for some $\beta\geq 1$ to be fixed later. By $\delta_{z_0}$ we denote the Dirac-measure at $z_0\in\C$. Let us define
  \begin{align*}
   \mu = \frac{\tau}{R} \sum_{j=0}^k q^j\delta_{w+A^{-1}q^j}.
  \end{align*}
  The constant $\tau>0$ will be chosen later. Before we go on we state a simple lemma which will be frequently applied in the following.
  \begin{lemma}
   Let $n>0$ be a real number. The function $s\mapsto s^n e^{-s}$ has a unique maximum on $\R_+$. Before this maximum the function is strictly increasing and after that maximum it is strictly decreasing.
  \end{lemma}
  One can prove the lemma by simply taking the derivative of the function. 
  
  \paragraph{\textbf{Part 1:} Estimation of $\LT\mu$.} We distinguish the two cases $t\leq A$ and $t>A$.
  
  \emph{Case 1: $t\leq A$.} We calculate
  
  \begin{align*}
   \LT\mu(t) &= \frac{\tau}{R}\sum_{j=0}^k q^j e^{t(w+A^{-1}q^j)} \\
   &= \frac{\tau}{R} e^{tw} \sum_{m=0}^{\infty} \frac{1}{m!}\left( \frac{t}{A} \right)^m \sum_{j=0}^k q^{(m+1)j} \\
   &= \frac{\tau}{R} \cdot e^{tw} \frac{(k+1)t^k}{A^k k!} \cdot \sum_{n=1}^{\infty}\frac{k!}{(n(k+1)-1)!} \left( \frac{t}{A} \right)^{(n-1)(k+1)} \\
   &=: \frac{\tau}{R} \cdot I \cdot II.
  \end{align*}
  Clearly $II$ is bounded from below by $1$ and bounded from above by a constant which does not depend on $k$ or $A$. Thus by Stirling's formula we get
  \begin{align*}
   \LT\mu(t) \geq c\frac{\tau}{R} \sqrt{k} e^{-\delta t} \left( \frac{e \delta t}{\delta A k} \right)^k .
  \end{align*}
  As a function in $t$ we can maximize the right-hand side by setting $\delta t=k$. If we furthermore define
  \begin{align}\label{eq: definition of tau}
   \tau = \frac{1}{\sqrt{k}}(\delta A)^k
  \end{align}
  we see that (\ref{eq: optimality 4}) is proved. Since $II$ is bounded from above we have
  \begin{align}\label{eq: LT from above}
   \LT\mu(t) \leq C\frac{\tau}{R} \sqrt{k} e^{-\delta t} \left( \frac{e \delta t}{\delta A k} \right)^k .
  \end{align}
  Again we maximize the right-hand side by setting $\delta t = k$ and plugging in (\ref{eq: definition of tau}). This leads to
  \begin{align*}
   \LT\mu(t) \leq C\frac{\tau}{R} \sqrt{k} e^{-k} \left( \frac{e}{\delta A} \right)^k  \leq \frac{C}{R}
  \end{align*}
  For $t\in[k/2\delta,2k/\delta]$ this is already what we want to have in (\ref{eq: optimality 3}).

  \emph{Case 1.1: $\delta t \leq k/2$.} In this case the maximum in (\ref{eq: LT from above}) with respect to $t$ is attained for $\delta t=k/2$. This yields
  \begin{align*}
   \abs{\LT\mu(t)} \leq C\frac{\tau}{R} \sqrt{k} e^{-\frac{k}{2}} \left( \frac{e}{2\delta A} \right)^k  
   = \frac{C}{R} \left( \frac{e}{4} \right)^{\frac{k}{2}} \leq \frac{\varepsilon}{R}
  \end{align*}
  The last inequality holds for sufficiently large $k$. We proved (\ref{eq: optimality 3}) for $\delta t \leq k/2$.

  \emph{Case 1.2: $2k \leq \delta t \leq \delta A$.} Condition (i) from Theorem \ref{thm: Optimality} yields $M_K^{-1}(c_1 t)\leq e^{\delta t/\alpha}$ for any $\alpha>0$ as long as $t$ is large enough. Thus, if we multiply (\ref{eq: LT from above}) by $M_K^{-1}(c_1 t)$ we get
  \begin{align*}
   M_K^{-1}(c_1 t)\abs{\LT\mu(t)} &\leq C\frac{\tau}{R} \sqrt{k} e^{-(1-\frac{1}{\alpha})\delta t} \left( \frac{e \delta t}{\delta A k} \right)^k \\
   &\leq \frac{C\sqrt{k}}{R} \left( \frac{2}{e^{1-\frac{2}{\alpha}}} \right)^{k} \leq \varepsilon
  \end{align*}
  for sufficiently large $k$. From the first to the second line we used that the maximum of the right-hand side of the first line is attained at $\delta t = 2k$ if $\alpha\geq2$. In the last estimate we used $e^{1-\frac{2}{\alpha}}>2$ which is true if $\alpha$ is large enough. We proved (\ref{eq: optimality 3}) for $2k\leq \delta t\leq \delta A$.
  
  \emph{Case 2: $t>A$.} Then we have
  \begin{align*}
   \abs{\LT\mu(t)} &\leq \frac{\tau}{R}(k+1)e^{-(\delta-A^{-1})t} \\
   &\leq \frac{C}{R} \sqrt{k} (\delta A)^k e^{-\delta A} e^{-(\delta-A^{-1})(t-A)}
  \end{align*}
  In the following we assume that $\delta-A^{-1}>0$ which is true for large $k$. 
  
  \emph{Case 2.1: $A<t<2A$.} In this case (using again $M_K^{-1}(c_1t)\leq e^{\delta t/\alpha}$ for large $t$) we get
  \begin{align*}
   M_K^{-1}(2c_1 A)\abs{\LT\mu(t)} &\leq \frac{C}{R} \sqrt{k} \left( kl(k) e^{-l(k)} \right)^k e^{\frac{2kl(k)}{\alpha}} \\
   &= \frac{C}{R} \sqrt{k} \left( kl(k) e^{-(1-\frac{2}{\alpha})l(k)} \right)^k \leq \varepsilon
  \end{align*}
  if we choose $\beta>1$ and let $\alpha$ satisfy $(1-\frac{2}{\alpha})^{-1}<\beta$ and if $k$ is large enough. We proved (\ref{eq: optimality 3}) for $A<t<2A$.
 
   \emph{Case 2.2: $t\geq 2A$.} If we use $\sqrt{k}(\delta A)^ke^{-\delta A}\leq 1$ for large $k$ we can calculate for an $\alpha>4$
  \begin{align*}
   M_K^{-1}(c_1 t)\abs{\LT\mu(t)} &\leq \frac{C}{R} e^{-(1-\frac{1}{kl(k)})(\delta t-\delta A)} e^{\frac{\delta t}{\alpha}} \\
   &\leq \frac{C}{R} e^{(\frac{1}{\alpha}-\frac{1}{4})\delta t} \leq \varepsilon .
  \end{align*}
  This finishes the proof of (\ref{eq: optimality 3}).
 
  \paragraph{\textbf{Part 2:} Estimation of $\CT\mu$.} First observe that as long as $z$ is no $(k+1)$-th root of unity we have
  \begin{align*}
   \sum_{j=0}^k \frac{q^j}{z-q^j} = \frac{k+1}{z^{k+1} - 1}.
  \end{align*}
  Clearly this equation must hold for some $k$-th order polynomial $p$ if one replace the term $k+1$ on the right-hand side by $p(z)$. Moreover the left-hand side is invariant under the substitution which replaces $z$ by $q z$. Thus $p(z)=p(qz)$. But this implies that $p$ is a constant. By plugging in $z=0$ we see that $p=k+1$.
 
  The observation yields for $z\in\Omega_M$
  \begin{align}\label{eq: nice formula for CTmu}
   \CT\mu(z) = \frac{\tau}{R} \frac{(k+1)A}{(A(z-w))^{k+1} - 1}.
  \end{align}
  Now it is not difficult to prove (\ref{eq: optimality 1}) for $\abs{\Im z - R}>2\delta$. The latter condition implies $\abs{z-w}>2\delta$. Thus, using (\ref{eq: nice formula for CTmu}) we get for $\abs{\Im z - R}>2\delta$ and $k$ large:
  \begin{align*}
   \abs{\CT\mu(z)} \leq C \frac{\tau}{R} kA (2\delta A)^{-k-1} \leq \frac{C\sqrt{k}}{\delta R} 2^{-k} \leq \varepsilon .
  \end{align*}
  If we don't have $\abs{\Im z - R}>2\delta$ we can merely estimate $\abs{z-w}\geq \delta - 1/M(\Im z)$. This yields for $z\in\Omega_M$ with $\abs{\Im z - R}>2\delta$ and for all $\gamma_1>1$
  \begin{align*}
   \abs{\CT\mu(z)} &\leq C \frac{\tau}{R} kA (\delta A(1-\frac{1}{\delta M(\Im z)}))^{-k-1} \\
   &\leq \frac{C\sqrt{k}}{\delta R} e^{\gamma_1\frac{k}{\delta M(\Im z)}} \\
   &\leq \frac{C\sqrt{k}}{\delta R} e^{\gamma_1 N(2\delta)\frac{k}{\delta M(R)}} \\
   &\leq \frac{C}{\delta R} \sqrt{M_K(R)} K(R)^{\frac{\gamma_1 N(2\delta)}{c_1}}.
  \end{align*}
  From the first to the second line we use the inequality $1-x\geq e^{-\gamma_1 x}$ which is valid for small $x\geq0$. If $M$ is bounded we choose $\delta$ large enough to make use of this inequality. From the second to the third line we used condition (ii) from Theorem \ref{thm: Optimality}. Choosing $\gamma=\gamma_1 N(2\delta)$ we get (\ref{eq: optimality 1}). Concerning Remark \ref{rem: Y} a reverse inequality for $\Im z = R$ can be proved analogously but in an even simpler way by using the inequality $1-x\leq e^{-x}$ which is valid for all $x\geq0$.
  
  \paragraph{\textbf{Part 3:} Estimation of $\LT'\mu$.} Finally we want to estimate the derivative of $\LT\mu$. 
  
  \emph{Case 1: $t\geq A$.} In this case we directly get for large $k$
  \begin{align*}
   \abs{\LT'\mu(t)} &\leq \frac{\tau}{R} (k+1) (R+A^{-1}) e^{-(\delta-A^{-1})t} \\
   &\leq C \frac{\sqrt{k}}{R} (\delta A)^k R e^{-\delta A} \leq \varepsilon .
  \end{align*}
  
  \emph{Case 2: $t<A$.} Let us first get a different representation of $\LT\mu$:
  \begin{align*}
   \LT'\mu(t) &= \frac{\tau}{R} \sum_{j=0}^k q^j(w+A^{-1}q^j) e^{(w+A^{-1}q^j)t} \\
   &= \frac{\tau}{R} e^{tw} \sum_{m=0}^{\infty} \frac{1}{m!} \left(\frac{t}{A}\right)^m \sum_{j=0}^k (wq^{(m+1)j}+A^{-1}q^{(m+2)j}) \\
   &= \frac{w}{R} \tau e^{tw} \frac{(k+1)t^k}{A^k k!} \sum_{n=1}^{\infty} \frac{k!}{(n(k+1)-1)!} \left( \frac{t}{A} \right)^{(n-1)(k+1)}
      \left[ 1 + \frac{n(k+1)-1}{wt} \right] .
  \end{align*}
  Note that if $t>t_0>0$ the sum at the end of the calculation is bounded by a constant which only depends on $t_0$.
  \begin{align}\nonumber
   \abs{\LT'\mu(t)} &\leq C\tau \sqrt{k} e^{-\delta t} \left( \frac{e \delta t}{\delta A k} \right)^k \left[ 1 + \frac{k}{Rt}\right] \\ \label{eq: LT' estimate}
   &\leq C e^{-\delta t} \left( \frac{e \delta t}{k} \right)^k \left[ 1 + \frac{k}{Rt}\right] 
  \end{align}
  Note that (\ref{eq: LT' estimate}) as a function in $t$ is increasing for $\delta t < k-1$ and decreasing for $\delta t > k$. Therefore we see that $\abs{\LT'\mu(t)}$ bounded by a constant not depending on $t$. This shows (\ref{eq: optimality 2}) for $k/2\delta \leq t \leq 2k/\delta$.

  \emph{Case 2.1: $\delta t \leq k/2$.} The maximum in (\ref{eq: LT' estimate}) is then attained for $\delta t = k/2$. This yields
  \begin{align*}
   \abs{\LT'\mu(t)} &\leq C e^{-\frac{k}{2}} \left( \frac{e}{2} \right)^k \leq C \left( \frac{e}{4} \right)^{\frac{k}{2}} \leq \varepsilon
  \end{align*}
  if $k$ is large enough.
  
    \emph{Case 2.2: $2k\leq \delta t \leq A$.} The maximum in (\ref{eq: LT' estimate}) is then attained for $\delta t = 2k$. This yields
  \begin{align*}
   \abs{\LT'\mu(t)} &\leq C e^{-2k} \left( 2e \right)^k \leq C \left( \frac{2}{e} \right)^{\frac{k}{2}} \leq \varepsilon
  \end{align*}
  if $k$ is large enough. This finishes the proof of Lemma \ref{lem: Optimality}.
 \end{proof}

 \begin{proof}[Proof of Theorem \ref{thm: Optimality}]
  For an $\varepsilon_0>0$ to be chosen later we define a sequence $(\varepsilon_n)$ by $\varepsilon_n=2^{-n}\varepsilon_0$. There exists a $\delta>0$, an increasing sequence of natural numbers $(k_n)$ and a sequence of measures $(\mu_n)$ according to Lemma \ref{lem: Optimality}. We may assume that $([R_n-2\delta,R_n+2\delta])$ and $([k_n/2\delta,2k_n/\delta])$ are sequences of pairwise disjoint intervals. Let us define
  \begin{align*}
   f(t) = \sum_{n=1}^{\infty} \LT\mu_n(t) \text{ for } t\geq0.
  \end{align*}
  The sum is uniformly convergent because of (\ref{eq: optimality 3}). The function $f$ is therefore continuous and since the sequence of derivatives converges uniformly (by (\ref{eq: optimality 2})) we see that $f$ has a bounded weak derivative given by
  \begin{align*}
   f'(t) = \sum_{n=1}^{\infty} \LT'\mu_n(t) \text{ for } t\geq0.
  \end{align*}
  By a similar argument the Laplace transform has the form
  \begin{align*}
   \fhat(z) = \sum_{n=1}^{\infty} \CT\mu_n(z)  \text{ for } z\in\Omega_M.
  \end{align*}
  Here the sum converges uniformly on compact subsets of $\overline{\Omega_M\cup\C_+}$ (by (\ref{eq: optimality 1})). We already know that the derivative of $f$ is bounded. The estimate (\ref{eq: optimality fhat estimate}) follows immediately from (\ref{eq: optimality 1}). It remains to prove (\ref{eq: optimality decay rate}). Let us set $t_n=k_n/\delta$ then we deduce from (\ref{eq: optimality 3}) and (\ref{eq: optimality 4}) that
  \begin{align*}
   \abs{f(t_n)} &\geq \frac{c}{R_n} - \varepsilon_0 \sum_{j\neq n} \frac{2^{-j}}{\max\{R_j, M_K^{-1}(c_1 t_n)\}} \\
   &\geq \frac{c}{R_n} - \varepsilon_0 \sum_{j\neq n} \frac{2^{-j}}{R_n} \\
   &\geq \frac{c}{R_n} = \frac{c}{M_K^{-1}(c_1 t_n)}.
  \end{align*}
  In the last line we chose $\varepsilon_0$ small enough.
 \end{proof}

\begin{remark}\label{rem: m}
 By the same technique one can also prove the optimality of Theorem \ref{thm: M_K theorem - Laplace variant} for $m>1$. To achieve this one just has to define the measure $\mu$ in Lemma \ref{lem: Optimality} by $\mu = \tau R^{-m} \sum_{j=0}^k q^j\delta_{w+A^{-1}q^j}$.
\end{remark}

\begin{remark}\label{rem: YY}
 With the help of remark \ref{rem: Y} one easily sees that for $\Im z = R_n$ the inequality (\ref{eq: optimality fhat estimate}) holds also in the reverse direction (for a different constant $C$).
\end{remark}

\subsection{On the optimality of the constant $c_1$ in Theorem \ref{thm: M_K theorem - Laplace variant}}\label{sec: Optimality of c1}
The literature seems not to pay much attention to the constant $c_1$ appearing in Theorem \ref{thm: M_K theorem - Laplace variant}. If we are interested in polynomial decay the constant does not influence the decay rate much. However, if for example $M_K^{-1}(t)=\exp(t^{\alpha})$ for some $\alpha\in(0,1]$ we immediately see that $c_1$ influences the decay rate in a crucial way. 
The aim of this subsection is to give a partial answer concerning the question of the optimality of $c_1$. Under not too restrictive conditions on $M$ and $K$ we show that Theorem \ref{thm: M_K theorem - Laplace variant} is valid for any $c_1<1$ and false for $c_1>1$. Unfortunately we have to exclude the important special case of exponential decay from our discussion.

\begin{theorem}\label{thm: Optimality c1}
 Let $p=\infty$. (a) In addition to the assumptions in Theorem \ref{thm: M_K theorem - Laplace variant} assume that $K$ increases faster than any polynomial and assume that $K(s)\geq c(1+s)^{-m}M(2s)^2$. Then (\ref{eq: rate}) holds for all $c_1<1$. (b) Let $M,K$ satisfy the assumptions of Theorem \ref{thm: M_K theorem - Laplace variant}. Assume in addition that for some $\gamma_0\geq1$
 \begin{align}\label{eq: additional condition on M}
  \forall s_1>0 \exists s_0>0 \forall s\geq s_0,s'\leq s_1: M(s+s')\leq \gamma_0 M(s).
 \end{align}
 Assume furthermore that $K$ increases faster than any polynomial in $sM(s)$. Let $c_1>\gamma_0$. Then there exists a locally integrable function $f:\R_+\rightarrow\C$, satisfying the assumptions of Theorem \ref{thm: M_K theorem - Laplace variant} such that (\ref{eq: rate}) does not hold for this choice of $c_1$.
\end{theorem}

\begin{remark}
 It is not difficult to find functions $M$ which satisfy (\ref{eq: additional condition on M}) for \emph{any} $\gamma_0>1$. Take for example $M$ to be a constant, a logarithm or a polynomial. It is also possible to take $M(s)=\exp(s^{\alpha})$ for $\alpha\in(0,1)$. On the other hand the example $M(s)=\exp(s)$ does not satisfy this condition for any $\gamma>1$. 
\end{remark}

\begin{remark}
 We think that the condition that $K$ increases faster than a polynomial in $s$ is natural in both parts of the theorem. On the other hand we don't know whether the growth condition on $K$ in terms of $M(s)$ or $M(2s)$ is a necessary assumption for the conclusion of Theorem \ref{thm: Optimality c1} to hold. Concerning (a) this condition is only necessary in the proof since we do not know whether Lemma \ref{lem: Fourier vs Laplace} is valid for $C_f=0$. Concerning (b) we need it because of the factor $M(\abs{\Im z})^{1/2}$ appearing in (\ref{eq: optimality fhat estimate}).
\end{remark}

\begin{proof}
 (a) The claim is proved by having a look into the proof of Theorem \ref{thm: M_K theorem - Laplace variant}. It is not difficult to see that in (\ref{eq: no observation}) is true for any $C_2>1$. To get (\ref{eq: J2 final estimate}) one has to choose $c_1$ in such a way that $K(R)^{\frac{1}{c_1C_2}-1}\geq cR^{m+1}$. Since $K$ grows super-polynomially in $s$ this means $c_1<1/C_2$. Now observe that in in the final step of the proof in Section \ref{sec: proof of Theorem Laplace}, before the proof of Lemma \ref{lem: Fourier vs Laplace}, one can choose any $c<1$. Here we use that $K(s) \geq c(1+s)^{-m}M(2s)^2$. Since $C_2$ can be chosen arbitrary close to $1$ the first assertion is proved.
 
 (b) Let $\gamma_0<\gamma<c_1$. First observe that the assumptions  of Theorem \ref{thm: Optimality} (including (ii')) are satisfied (concerning $m>1$ see also Remark \ref{rem: m}). Thus there exists a locally integrable function $f:\R_+\rightarrow\C$ such that the conclusion of Theorem \ref{thm: Optimality} is satisfied. Since $K$ grows faster than any polynomial of $M(s)$ we can withdraw the factor $M(\abs{\Im z})^{1/2}$ from (\ref{eq: optimality fhat estimate}) if we replace $\gamma/c_1$ by $1$ in this inequality. Now the function satisfies the assumptions of Theorem \ref{thm: M_K theorem - Laplace variant} but it fails to satisfy (\ref{eq: rate}) for our choice of $c_1$ by Theorem \ref{thm: Optimality}.
\end{proof}

 \section{Application: Local decay rates}\label{sec: Applications: decay of waves}
 Our results can be applied to calculate local decay rates for $C_0$-semigroups. To fix some of our notation let $T=(T(t))_{t\geq0}$ be a $C_0$-semigroup on a Banach space $X$ with generator $A:D(A)\rightarrow X$. We denote 
 \begin{equation}\nonumber
  \omega_0(T) = \inf \left\{ \omega\in\R; (t\mapsto \norm{e^{-\omega t}T(t)}) \text{ is bounded on $\R_+$} \right\}. 
 \end{equation}
 In Subsection \ref{sec: Local energy decay} we apply the abstract setting from Subsection \ref{sec: Decay of C0-semigroups} to local energy decay for the wave equation in an odd-dimensional exterior domain. In Subsection \ref{sec: Local energy decay} we naturally restrict our considerations to the case $p=\infty$. A discussion of $L^p$-rates for semigroups and an application to the wave equation can be found in \cite[Section 6]{BattyBorichevTomilov2016}.
 
  \subsection{Local decay of $C_0$-semigroups}\label{sec: Decay of C0-semigroups}
  The following is an immediate consequence of our main result Theorem \ref{thm: M_K theorem - Laplace variant}.
  \begin{corollary}[to Theorem \ref{thm: M_K theorem - Laplace variant}]\label{cor: M_K theorem - Laplace variant}
   Let $T$ be a $C_0$-semigroup on a Banach space $(X,\norm{\cdot})$ with generator $A$ and $\omega_0(T)\geq0$. Let $P_1$ and $P_2$ be two bounded operators on $X$, let $x\in X$ and let $1<p\leq\infty$. 
   Let $M, K:\R_+\rightarrow [2,\infty)$ be continuous and increasing functions satisfying
   \begin{enumerate}
    \item[(i)] $\forall s>1: K(s) \geq \max\{s, M(s)\}$,
    \item[(ii)] $\exists\varepsilon\in(0,1): K(s) = O\left(e^{e^{(sM(s))^{1-\varepsilon}}}\right)$ as $s\rightarrow\infty$.
   \end{enumerate}
   Let $G(z)=P_2(z-A)^{-1}P_1 x$ for $\Re z > 0$. Assume that $G$ extends analytically to the domain $\Omega_M\cup\C_+$ and satisfies the estimate
   \begin{equation}\label{eq: infty resolvent growth}
    \norm{G(z)} \leq K(\abs{\Im z}) \text{ for } z\in \Omega_M .
   \end{equation}
   Assume furthermore that $(t\mapsto \norm{P_2T(t)P_1x})\in L^p(\R_+)$. Then for all $m\in\N_1$ and $\omega > \omega_0(T)$ we have
   \begin{equation}\nonumber
    (t\mapsto w_{M_K}(t)^m\norm{P_2 T(t) (\omega-A)^{-m} P_1 x}) \in L^p(\R_+) 
   \end{equation}
   where $M_K(s)=M(s)\log(K(s))$.
  \end{corollary}
  
  \begin{remark}
   Observe that the condition $(t\mapsto \norm{P_2T(t)P_1x})\in L^p(\R_+)$ is trivially satisfied if $T$ is a bounded $C_0$-semigroup and $p=\infty$. If in this case we also have that $A$ is invertible then - as is clear from the proof - one can also take $\omega=0$. In in the case $P_1=P_2=1$ we note that if $p\neq\infty$ and $(t\mapsto \norm{T(t)x})\in L^p(\R_+)$ is true \emph{for all} $x\in X$ then by Datko's theorem (see e.g. \cite[Theorem 5.1.2]{ArendtBattyHieberNeubrander} the semigroup is automatically exponentially stable.
  \end{remark}
  \begin{remark}
   In the particular case $P_1=P_2=1$ one typically assumes that the resolvent extends continuously to the imaginary axis and satisfies an estimate $\norm{(is-A)^{-1}}\leq M(\abs{s})$ for $s\in\R$. This then implies that the resolvent extends analytically to $\Omega_M$ and it satisfies (\ref{eq: infty resolvent growth}) with $K$ being a multiple of $M$ in a slightly smaller domain. So in this situation our corollary does not improve known results. 
   
   However, our main interest in applying this theorem is to consider the case where $P_1$ and $P_2$ are not the identity. We think that a typical situation is that $M$ is a slowly increasing function (possibly constant) and $K$ is a (possibly much) faster increasing function. That is, we assume that the perturbed resolvent extends to a relatively large domain to the left of the imaginary axis, but only has to satisfy a mild growth condition. We illustrate this philosophy in Subsection \ref{sec: Local energy decay}.
  \end{remark}

  \begin{proof}
   Let us define $f(t)=P_2 T(t) (\omega-A)^{-m} P_1 x$. Then we have for $t>0$ and for $z\in\Omega_M$
   \begin{align*}
    f^{(m)}(t) = P_2 T(t) [\omega(\omega-A)^{-1} - 1]^m P_1 x \text{ and } \\
    \fhat(z) = \sum_{j=0}^{m-1} (\omega-z)^{-(j+1)}P_2(\omega-A)^{-(m-j)}P_1 x + (\omega-z)^{-m}G(z) .
   \end{align*}
   The second line immediately implies (\ref{eq: fhat condition}) up to a constant factor. The first line implies $\norm{f^{(m)}}\in L^p(\R_+)$ since
   \begin{align*}
    \norm{P_2T(t)(\omega-A)^{-1}P_1x} = \norm{\int_{0}^{\infty} P_2 e^{-\omega\tau} T(t+\tau) P_1 x d\tau} \leq \norm{P_2T(t)P_1x} .
   \end{align*}
   Thus the conclusion of the corollary follows from Theorem \ref{thm: M_K theorem - Laplace variant}.
  \end{proof}

  \subsection{Local energy decay for waves in exterior domains}\label{sec: Local energy decay}
  We want to show that Corollary \ref{cor: M_K theorem - Laplace variant} applies naturally to local energy decay for waves in exterior domains. It improves known decay rates and even simplifies the proofs. 
  
  Let $\Omega\subsetneqq\R^d$ be a connected open set with bounded complement and non-empty $C^{\infty}$-boundary. The dimension $d$ is assumed to be at least $2$. We consider the wave equation on this domain:
  \begin{equation}\label{eq: exterior wave equation}
   \left\{
   \begin{array}{lr}
    u_{tt}(t,x) - \Delta u(t,x) = 0 &  (t\in(0,\infty), x\in\Omega), \\
    u(t,x) = 0 & (t\in(0,\infty), x\in\partial\Omega), \\
    u(0,x) = u_0(x), u_t(0,x) = u_1(x) & (x\in\Omega) .
    \end{array}
   \right.
  \end{equation}
  Let us fix a radius $r>0$ such that the obstacle $K=\R^d\backslash\Omega$ is included in the open ball $B_r$ of radius $r$ and center $0$. We define a \emph{state} (at time $t$) of the system by $\xnice(t):=(u,v)(t):=(u(t),u_t(t))$. We define the local energy of a state by
  \begin{equation}\label{eq: local energy}
   E^{loc}(\xnice) = \int_{\Omega\cap B_r} \abs{\nabla u}^2 + \abs{v}^2 dx .
  \end{equation}
  Clearly equation (\ref{eq: local energy}) is well defined for all $u\in C_c^{\infty}(\Omega)$ and $v\in L^2(\Omega)$. Therefore it is also well defined on the \emph{energy space}
  \begin{equation}\nonumber
   \Ho = H_D^1(\Omega) \times L^2(\Omega),
  \end{equation}
  where $H_D^1(\Omega)$ is the completion of $C_c^{\infty}(\Omega)$ under the norm given by the quadratic form $u\mapsto\int_{\Omega}\abs{\nabla u}^2$.
  
  The wave equation (\ref{eq: exterior wave equation}) on the energy space $\Ho$ can be reformulated in the language of $C_0$-semigroups. Therefore we write $\xnice(t)=(u(t),u_t(t))$ set $\xnice_0=(u_0,u_1)$ and write
  \begin{equation}\label{eq: exterior CP}
    \left\{
    \begin{aligned}
    \dot{\xnice}(t) = \A \xnice(t) , \\
    \xnice(0) = \xnice_0 \in \Ho 
   \end{aligned}
   \right.
  \text{ where }
   \A = \left(
   \begin{array}{cc}
    0       & 1 \\
    \Delta_D  & 0     \\
   \end{array}\right)
   \text{ with }
   D(\A) = D(\Delta_D)\times H^1_D(\Omega) .
  \end{equation}
  The Dirichlet-Laplace operator $\Delta_D$ has the domain $D(\Delta_D)=\{u\in H^1_D(\Omega); \Delta u\in L^2(\Omega)\}$, where $\Delta$ denotes the Laplace operator in the sense of distributions. It can be proved that the wave operator $\A$ is skew-adjoint (see e.g. \cite[Theorem V.1.2]{LaxPhillips1967}). Therefore the following theorem follows by Stone's theorem (see e.g. \cite[Appendix 1, Theorem 2]{LaxPhillips1967}).
  \begin{theorem}
   The wave operator $\A$ generates a unitary $C_0$-group on $\Ho$.
  \end{theorem}

  Let $m\in\N_0$. We are interested in the uniform decay rate of the local energy with respect to sufficiently smooth initial data, compactly supported in the ball of radius $r$:
  \begin{equation}\label{eq: pm}
   p_m(t) := \sup \left\{ \left( \frac{E^{loc}(\xnice(t))}{\norm{\xnice_0}^2_{H^{m+1}\times H^m}} \right)^{\frac{1}{2}} ; \xnice_0\in H^{m+1}_{\text{comp}}\times H^m_{\text{comp}}(\overline{\Omega}\cap B_r) \right\}.
  \end{equation}
  Here by $H^m_{\text{comp}}(\overline{\Omega}\cap B_r)$ we denote all square-integrable functions, supported on $\overline{\Omega}\cap B_r$ for which all weak derivatives up to order $m$ are square-integrable too. It is well known that $p_0$ either does not decay to zero, or decays exponentially for $d$ odd and as $t^{-d}$ for $d$ even. Moreover the decay can be characterized by boundedness of the local resolvent of $\A$ on the imaginary axis. We refer to \cite{Vodev1999} and references therein for these facts. 
  
  In the following we assume $m\in\N_1$. Following the philosophy of the present article we see that we have to investigate the resolvent of $\A$. In the literature on local energy decay it is common to investigate the \emph{outgoing resolvent} of the stationary wave equation. For $\Re z>0$ and $f\in L^2(\Omega)$ the outgoing resolvent is defined as a Laplace transform:
  \begin{align*}
   R(z)f = \int_0^{\infty} e^{-zt} u(t) dt
  \end{align*}
  where $u$ is the first component of the solution to (\ref{eq: exterior CP}) for $x_0=(0,f)$. It is not difficult to show that $w=R(z)f$ then satisfies the stationary wave equation

  \begin{equation}\label{eq: stationary exterior wave equation}
   \left\{
   \begin{array}{lr}
    z^2 w(x) - \Delta w(x) = f(x) &  (x\in\Omega), \\
    w(x) = 0 & (x\in\partial\Omega). \\
    \end{array}
   \right.
  \end{equation}

  There is an important relation between this operator and the resolvent of $\A$: For $\Re z > 0$ we have
  \begin{equation}\label{eq: resolvents R vs A}
      (z-\A)^{-1} = \left(
   \begin{array}{cc}
    zR(z)       & R(z)   \\
    z^2R(z)-1   & zR(z)  \\
   \end{array}\right) .
  \end{equation}
  Let us fix a cut-off function $\chi\in C_c^{\infty}(\R^d)$ with $0\leq \chi \leq 1$ such that $\chi=1$ on a neighbourhood of $K$. We define the truncated resolvent by $R_{\chi}(z)=\chi R(z) \chi$ where we consider $\chi$ as a multiplication operator on $L^2(\Omega)$. From the definition we see that the outgoing truncated resolvent is an analytic function in the interior of $\C_+$. The next proposition illuminates its behaviour on the other half of the complex plane.
  \begin{proposition}\label{prop: Burq Appendix B}
   (i)\text{\cite[Appendix B]{Burq1998}} The truncated outgoing resolvent $R_{\chi}$ extends analytically across $i\R\backslash\{0\}$. Moreover, for any $\varepsilon>0$ $R_{\chi}:L^2(\Omega)\rightarrow L^2(\Omega)$ is bounded in a small neighbourhood of $0$ intersected with a sector $\{ z\in\C\backslash\{0\}; \abs{\arg z - \pi} \geq \varepsilon \}$. (ii)\cite[Corollary V.3.3 together with Remark V.4.3]{LaxPhillips1967} If the dimension $d$ is odd $R_{\chi}$ extends meromorphically to $\C$.
  \end{proposition}
  In the following we want to restrict our considerations to the odd-dimensional (i.e. $d$ is odd) case only. By Proposition \ref{prop: Burq Appendix B} together with (\ref{eq: resolvents R vs A}) we immediately see that $(z\mapsto\chi(z-\A)^{-1}\chi)$ for $\Re z>0$ extends to a meromorphic function $G_{\chi}$ on $\C$ which has no poles on $i\R$. Here we consider $\chi$ as an operator on $\Ho$ acting as $\chi(u_0,u_1)=(\chi u_0, \chi u_1)$. Since the spectrum of $\A$ is the entire imaginary axis the equality $G_{\chi}(z)= \chi(z-\A)^{-1}\chi$ \emph{does not hold} for $\Re z<0$ in general. 
  
  The following proposition is well-known in the literature on exterior wave equations. The proof is not difficult but rather lengthy. Unfortunately we could not find a proof in the literature but we refer to \cite[Section]{BurqHitrik2007} for a similar statement and the idea of the proof.  

  \begin{proposition}\label{prop: resolvent A vs R}
   Let $\delta>0$ and let $\tilde{\chi}$ be defined as $\chi$ but with $\tilde{\chi}=1$ on the support of $\chi$. Let $z$ with $-\delta<\Re z < 0$ be no pole of $R_{\chi}$, then
   \begin{equation}\nonumber
    \norm{G_\chi(z)} \leq C (1+\abs{\Im z}) \norm{R_{\tilde{\chi}}(z)}_{L^2\rightarrow L^2}
   \end{equation}
   holds with a constant $C>0$ independent of $z$. The reverse inequality - with a different constant and $\tilde{\chi}$ replaced by $\chi$ - is also true.
  \end{proposition}
  
  It can happen that a whole strip $\{ z\in\C ; -\delta < \Re z < 0 \}$ is free of poles - see for instance \cite{Ikawa1988}. In \cite{BonyPetkov2006} the impact of the presence of such a strip on local energy decay was studied. There it was shown in a first step that such a strip implies that the norm of $G_{\chi}$ can be estimated by $C\exp(C\abs{\Im(z)}^{\alpha})$ on this strip for some $\alpha\geq1$. Indeed $\alpha=d-1$ in this article but it was not shown that this is optimal. In a second step the authors showed that this implies a bound of the form $(1+\abs{\Im z})^{\alpha}$ on $G_{\chi}$ in a region of the form $\{ z\in\C; -C(1+\abs{\Im z})^{-\alpha} < \Re z < 0 \}$. Finally in a third step they applied a Tauberain theorem (more precisely \cite[Proposition 1.4]{PopovVodev1999}) to get a $(\log(t)/t)^{1/\alpha}$ decay rate. 
  
  However, given a polynomial bound on the resolvent it would be desirable to have a polynomial decay of the local energy - \emph{without the logarithmic loss}! If we were not in a local situation then the results of \cite{BorichevTomilov2010} would help us to deduce our desired result without the logarithmic loss. Unfortunately it is not known whether \cite[Theorem 2.4]{BorichevTomilov2010} generalizes to local decay of semigroups on Hilbert spaces. In the following we show that with the help of Corollary \ref{cor: M_K theorem - Laplace variant} we get rid of the logarithmic loss. It even simplifies the proof in the sense that the second step is not necessary anymore since our preconditions in Corollary \ref{cor: M_K theorem - Laplace variant} are fulfilled by the local resolvent on the strip.

  By the preceding discussion it is reasonable to assume from now on the following conditions to be satisfied:
  \begin{enumerate}
   \item[(i)] There is a $\delta>0$ such that $R_{\chi}$ has no poles in $S_{\delta}=\{ z\in\C ; -\delta < \Re z < 0 \}$.
   \item[(ii)] There is a continuous and increasing function $\Mtilde:\R_+\rightarrow[2,\infty)$ satisfying $\Mtilde(s)\geq c\log(2+s)$ for any $s\geq0$ such that $\abs{\Im z} \norm{R_{\chi}(z)}_{L^2\rightarrow L^2} \leq C \exp(C \Mtilde(\abs{\Im z}))$ holds for all $z\in S_{\delta}$.
  \end{enumerate}
  Under these assumptions we can prove:
  \begin{theorem}
   Let $d$ be odd and let (i) and (ii) above be satisfied. Let $m\in\N_1$. Then
   \begin{equation}\nonumber
    p_m(t) \leq \frac{C}{\Mtilde^{-1}(c_1 t)^m}
   \end{equation}
   holds for a sufficiently small constant $c_1$ and a sufficiently large constant $C$. Here $\Mtilde^{-1}$ denotes the right-continuous right-inverse of $\Mtilde$.
  \end{theorem}
  \begin{proof}
   For $\Re z > 0$ let $G_{\chi}(z)=\chi (z-\A)^{-1} \chi$. Assumptions (i) and (ii) together with Proposition \ref{prop: resolvent A vs R} imply that $G_{\chi}$ extends analytically to $S_{\delta}\cup\overline{\C_+}$ and satisfies
   \begin{equation}\nonumber
    \norm{G_{\chi}(z)} \leq C \exp( C \Mtilde(\abs{\Im z}) ) \text{ for } z\in S_{\delta} .
   \end{equation}
   Thus by Corollary \ref{cor: M_K theorem - Laplace variant} (put $M=\delta$ and $K=C\exp\circ(C\Mtilde)$) we get (uniformly in $\xnice_0\in\Ho$)
   \begin{equation}
    \norm{\chi e^{t\A} (1-\A)^{-m} \chi \xnice_0} \leq \frac{C}{\Mtilde^{-1}_r(c_1 t)^m} \norm{\xnice_0} .
   \end{equation}
   The uniformity in $\Ho$ follows from the closed graph theorem. For simplicity we assume $m=1$ in the following. The general case can be treated almost the same way.
   
   Let $\chi_1\in C_c^{\infty}(\R^d)$ be a function such that $0\leq \chi_1 \leq 1$ and $\chi_1=1$ on $\spt \chi$. Of course Propositions \ref{prop: Burq Appendix B} and \ref{prop: resolvent A vs R} remain valid if one replaces $\chi$ by $\chi_1$. Note that the commutator $[\chi,1-\A]$ is a bounded operator on $\Ho$. Let $\xnice_1=(1-\A)^{-1}\xnice_0\in D(\A)$. Observe
   \begin{align*}
    \norm{\chi e^{t\A} \chi \xnice_1} 
    &\leq \norm{\chi e^{t\A} (1-\A)^{-1} \chi \xnice_0} + \norm{\chi (\chi_1 e^{t\A} (1-\A)^{-1} \chi_1) [\chi,(1-\A)] \xnice_1} \\
    &\leq \frac{C}{\Mtilde^{-1}_r(c_1 t)} (\norm{\xnice_0} + \norm{\xnice_1}) \\ 
    &\leq \frac{C}{\Mtilde^{-1}_r(c_1 t)} \norm{\xnice_1}_{D(\A)} .
   \end{align*}
  Without loss of generality we may assume that $\chi=1$ on $B_r$. Observe that the norm of elements of $D(\A)$, supported in $\overline{\Omega}\cap B_r$, is equivalent to the norm in the space $H^2\times H^1(\Omega)$. This follows from maximal regularity of the Dirichlet-Laplace operator on the bounded and smooth domain $\Omega\cap B_r$. Thus the last inequality (restricted to those $\xnice_1$ with support in $B_r$) implies the conclusion of the theorem.
  \end{proof}

 

 \subsection*{Acknowledgements} I am most grateful to Ralph Chill and Yuri Tomilov for valuable discussions on the topic of this article. I would like to thank the department of mathematics of the Nicolaus Copernicus University in Toru\'n for its hospitality. The idea to work on this topic came to me during a visit in december 2016.


\bibliographystyle{plainnat}
\bibliography{Refs}

\hfill

Technische Universit\"{a}t Dresden, Fachrichtung Mathematik, Institut f\"{u}r Analysis, 01062, Dresden, Germany. Email: \textit{Reinhard.Stahn@tu-dresden.de}

\end{document}